\documentclass{article}

\usepackage[english]{babel}
\usepackage{graphicx}
\usepackage{textcomp}
\usepackage{amssymb}
\usepackage{amsmath}
\usepackage{setspace}
\usepackage{geometry}
\usepackage{bm}
\usepackage{centernot}
\usepackage{hyperref}

\newtheorem{proposition}{Proposition}
\newtheorem{corollary}{Corollary}
\newtheorem{remark}{Remark}
\newtheorem{definition}{Definition}
\newtheorem{theorem}{Theorem}
\newtheorem{lemma}{Lemma}
\newtheorem{example}{Example}
\newenvironment{proof}[1][Proof:]{\begin{trivlist}
\item[\hskip \labelsep {\bfseries #1}]}{\end{trivlist}}

\newcommand{\id}{{\mathbf 1}}

\newcommand{\tr}{\operatorname{tr}}

\usepackage{xcolor}

\begin{document}
\title{Approximations in $L^1$ with convergent Fourier series}

\author{Zhirayr Avetisyan\footnote{Department of Mathematics, University of California, Santa Barbara. z.avetisyan@math.ucsb.edu} \and Martin Grigoryan\footnote{Chair of Higher Mathematics, Physics Faculty, Yerevan State University. gmarting@ysu.am} \and Michael Ruzhansky\footnote{Department of Mathematics,
  Ghent University, Belgium, and
  School of Mathematical Sciences,
    Queen Mary University of London,
     United Kingdom. ruzhansky@gmail.com
}}

\maketitle

\begin{abstract}
For a separable finite diffuse measure space $\mathcal{M}$ and an orthonormal basis $\{\varphi_n\}$ of $L^2(\mathcal{M})$ consisting of bounded functions $\varphi_n\in L^\infty(\mathcal{M})$, we find a measurable subset $E\subset\mathcal{M}$ of arbitrarily small complement $|\mathcal{M}\setminus E|<\epsilon$, such that every measurable function $f\in L^1(\mathcal{M})$ has an approximant $g\in L^1(\mathcal{M})$ with $g=f$ on $E$ and the Fourier series of $g$ converges to $g$, and a few further properties. The subset $E$ is universal in the sense that it does not depend on the function $f$ to be approximated. Further in the paper this result is adapted to the case of $\mathcal{M}=G/H$ being a homogeneous space of an infinite compact second countable Hausdorff group. As a useful illustration the case of $n$-spheres with spherical harmonics is discussed. The construction of the subset $E$ and approximant $g$ is sketched briefly at the end of the paper.
\end{abstract}

\section*{Introduction}

In the present paper we work with finite measure spaces $(\mathcal{M},\Sigma,\mu)$. For efficiency of nomenclature we will write $\mathcal{M}=(\mathcal{M},\Sigma,\mu)$ and $|A|=|A|_\mu=\mu(A)$ for every $A\in\Sigma$, where the $\sigma$-algebra $\Sigma$ and the measure $\mu$ are clear from the context. Consider a separable finite measure space $(\mathcal{M},\Sigma,\mu)$. Separability here simply means that all spaces $L^p(\mathcal{M})$ for $1\le p<\infty$ are separable. Let $\{\varphi_n\}_{n=1}^\infty$ be an orthonormal basis of $L^2(\mathcal{M})$ with $\varphi_n\in L^\infty(\mathcal{M})$ for all $n\in\mathbb{N}$. For a function $f\in L^1(\mathcal{M})$ we denote its Fourier components by
$$
Y_n(x;f)=c_n(f)\varphi_n(x),\quad c_n(f)=(f,\varphi_n)_2=\int_\mathcal{M}f(x)\varphi_n^*(x)d\mu(x),\quad\forall n\in\mathbb{N},
$$
where $\varphi_n^*$ denotes the complex conjugate of $\varphi_n$.
The possibly divergent Fourier series of $f$ will be
$$
\sum_{n=1}^\infty Y_n(x;f).
$$
Note that already for the trigonometric system on the interval there exists an integrable function of which the Fourier series diverges in $L^1$ (\cite{Bar64}, Chapter VIII, \S 22). We will often make use of Fourier polynomials and orthogonal series of the form
\begin{equation}
Q(x)=\sum_nY_n(x)=\sum_nc_n\varphi_n(x),\quad c_n\in\mathbb{C},\label{FourierPolynom}
\end{equation}
without reference to a particular function for which these may be the Fourier components. Denote by
\begin{equation}
\sigma(f)=\left\{n\in\mathbb{N}\,\vline\quad c_n(f)\neq0\right\},\quad f\in L^1(\mathcal{M}),
\end{equation}
the spectrum of a function $f$.

Before stating our main theorem let us recall the notion of diffuseness for a measure space.
\begin{definition} In a measure space $(\mathcal{M},\Sigma,\mu)$, a measurable subset $A\in\Sigma$ is called an atom if
$|A|>0$ and for every $B\in\Sigma$ with $B\subseteq A$ either $|B|=|A|$ or $|B|=0$. The measure space $(\mathcal{M},\Sigma,\mu)$ is called diffuse or non-atomic if it has no atoms.
\end{definition}

The main result of this paper is the following
\begin{theorem}\label{MainTheorem} Let $\mathcal{M}$ be a separable finite diffuse measure space, and let $\{\varphi_n\}_{n=1}^\infty$ be an orthonormal system in $L^2(\mathcal{M})$ consisting of bounded functions $\varphi_n\in L^\infty(\mathcal{M})$, $n\in\mathbb{N}$. For every $\epsilon,\delta>0$ there exists a measurable subset $E\in\Sigma$ with measure $|E|>|\mathcal{M}|-\delta$ and with the following property; for each function $f\in L^1(\mathcal{M})$ with $\|f\|_1>0$ there exists an approximating function $g\in L^1(\mathcal{M})$ that satisfies:

1. $\|f-g\|_1<\epsilon$,

2. $f=g$ on $E$,

3. the Fourier series of $g$ converges in $L^1(\mathcal{M})$,

4. we have
$$
\sup_m\left\|\sum_{n=1}^m Y_n(g)\right\|_1<2\min\left\{\|f\|_1,\|g\|_1\right\}.
$$
\end{theorem}
Luzin proved that every almost everywhere finite function $f$ on $[0,1]$ can be modified on a subset of arbitrarily small positive measure so that it becomes continuous. Further results in this direction were obtained by Menshov and others. See \cite{Luz12}, \cite{Men41}, \cite{Hel49}, \cite{Men51}, \cite{Pri69}, \cite{Osk76}, \cite{Kis79}, \cite{KaKo88}, \cite{Gri90}, \cite{Gri90a}, \cite{Gri91}, \cite{GrKr13}, \cite{GGK15}, \cite{GrGa16},  \cite{GrSa16} for earlier results in this direction for classical orthonormal systems. Let us note that if $\mathcal{M}$ is not diffuse (i.e., it has atoms) then Statement 4 of this theorem may not hold with any coefficient on the right hand side. This is illustrated in the next
\begin{example} For every natural $N\in\mathbb{N}$, let $(\mathcal{M},\Sigma,\mu)=(\mathbb{N}_2,2^{\mathbb{N}_2},P)$ be the probability space with orthonormal basis $\{\varphi_1,\varphi_2\}$ of $L^2(\mathcal{M})$, where
$$
\mathbb{N}_2=\{1,2\},\quad P(\{1\})=\frac3{16N^2-1},\quad\varphi_1=\left(2N,\frac12\right).
$$
Take $\delta=1/35$ and $f=(1,0)$. Then $|E|>1-\delta$ forces $E=\mathcal{M}$, and therefore $f=g$ on $E$ implies $f=g$ on $\mathcal{M}$. Now
$$
\left\|Y_1(g)\right\|_1=\left\|Y_1(f)\right\|_1=|c_1(f)|\|\varphi_1\|_1=\frac{3N[(8N+3)^2-25]}{4(16N^2-1)}>N\|f\|_1=\frac{3N}{16N^2-1}.
$$
\end{example}

\noindent Theorem \ref{MainTheorem} is equivalent to the following theorem, which can be obtained by repeatedly applying Theorem \ref{MainTheorem} with fixed $f\in L^1(\mathcal{M})$ and $\epsilon_m=\frac1m$, $\delta_m=\frac1m$, $m=1, 2, \ldots$.
\begin{theorem}\label{MainTheorem**} Let $\mathcal{M}$ be a separable finite diffuse measure space, and let $\{\varphi_n\}_{n=1}^\infty$ be an orthonormal system in $L^2(\mathcal{M})$ consisting of bounded functions $\varphi_n\in L^\infty(\mathcal{M})$, $n\in\mathbb{N}$. There exists an increasing sequence of subsets $\{E_m\}_{m=1}^\infty$, $E_m\subset E_{m+1}\subset\mathcal{M}$, with $\lim\limits_{m\to\infty}|E_m|=|\mathcal{M}|$, such that for every integrable function $f\in L^1(\mathcal{M})$ with $\|f\|_1>0$ there exists a sequence of approximating functions $\{g_m\}_{m=1}^\infty$, $g_m\in L^1(\mathcal{M})$, so that the following statements hold:

1. $g_m\xrightarrow[m\to\infty]{}f$ in $L^1(\mathcal{M})$,

2. $f=g_m$ on $E_m$, $\forall m\in\mathbb{N}$,

3. the Fourier series of $g_m$ converges in $L^1(\mathcal{M})$, $\forall m\in\mathbb{N}$,

4. we have
$$
\sup_{N}\left\|\sum_{n=1}^N Y_n(g_m)\right\|_1<2\min\left\{\|f\|_1,\|g_m\|_1\right\}, \quad\forall m\in\mathbb{N}.
$$
\end{theorem}

\begin{remark} Not for every orthonormal system $\{\varphi_n\}_{n=1}^\infty$ does an arbitrary integrable function $f\in L^1(\mathcal{M})$ have an orthogonal series $\sum_{n=1}^\infty Y_n$ of the form (\ref{FourierPolynom}) that converges to $f$ in $L^1(\mathcal{M})$, and if that happens then $\sum_{n=1}^\infty Y_n$ is necessarily the Fourier series of $f$, i.e., $Y_n=Y_n(f)$.
\end{remark}
For instance, in case of spherical harmonics this is guaranteed only in $L^2(\mathbb{S}^2)$ \cite{BoCl73}. However, the following weaker statement is a corollary of Theorem \ref{MainTheorem**} and holds true for all integrable functions.
\begin{corollary} Under the assumptions of Theorem \ref{MainTheorem**}, there exists an increasing sequence of subsets $\{E_m\}_{m=1}^\infty$, $E_m\subset E_{m+1}\subset\mathcal{M}$, such that $\lim\limits_{m\to\infty}|E_m|=|\mathcal{M}|$ with the following property. For any fixed integrable function $f\in L^1(\mathcal{M})$ with $\|f\|_1>0$ and for every natural number $m\in\mathbb{N}$ there is an orthogonal series $\sum_{n=1}^\infty Y_n^{(m)}$ of which the restriction $\sum_{n=1}^\infty Y_n^{(m)}|_{E_m}$ to the subset $E_m$ converges to the restriction $f|_{E_m}$ in $L^1(E_m)$. In $L^1(\mathcal{M})$ the series $\sum_{n=1}^\infty Y_n^{(m)}$ converges to a function $g_m\in L^1(\mathcal{M})$. The sequence of these functions $\{g_m\}_{m=1}^\infty$ converges to $f$ in $L^1(\mathcal{M})$.
\end{corollary}

\section*{The general case}

Theorem \ref{MainTheorem} is true for every finite separable diffuse measure space $\mathcal{M}$, but it will be more convenient to reduce the problem to that for a smaller class of measure spaces and then to prove the theorem for that class. First let us show that Theorem \ref{MainTheorem} is invariant under isomorphisms of measure algebras. For that purpose we will reformulate {Theorem \ref{MainTheorem}} in a way that makes no reference to the actual measure space $\mathcal{M}$ but only to its measure algebra $\mathcal{B}(\mathcal{M})$. We note that if we replace the set $E$ produced by \mbox{Theorem \ref{MainTheorem}} by another measurable set $E'\in\Sigma$ such that the symmetric difference is null, $|E\bigtriangleup E'|=0$, then all statements of the theorem remain valid with $E'$ instead of $E$. This brings us to the following equivalent formulation of {Theorem \ref{MainTheorem}}.

\begin{theorem}\label{MainTheorem*} Let $\mathcal{M}$ be a finite separable diffuse measure space, and let $\{\varphi_n\}_{n=1}^\infty$ be an orthonormal system in $L^2(\mathcal{M})$ consisting of bounded functions $\varphi_n\in L^\infty(\mathcal{M})$, $n\in\mathbb{N}$. For every $\epsilon,\delta>0$ there exists a function $\chi_E\in L^\infty(\mathcal{M})$ with $\chi_E^2=\chi_E$ and $\|\chi_E\|_1>|M|-\delta$, with the following properties; for each function $f\in L^1(\mathcal{M})$ with $\|f\|_1>0$ there exists an approximating function $g\in L^1(\mathcal{M})$ that satisfies:

1. $\|f-g\|_1<\epsilon$,

2. $(f-g)\chi_E=0$,

3. the Fourier series of $g$ converges in $L^1(\mathcal{M})$,

4. we have
$$
\sup_m\left\|\sum_{n=1}^m Y_n(g)\right\|_1<2\min\left\{\|f\|_1,\|g\|_1\right\}.
$$
\end{theorem}

In this form the theorem relies only upon spaces $L^p(\mathcal{M})$, $p=1,2,\infty$, which can be constructed purely out of the measure algebra $\mathcal{B}(\mathcal{M})$ with no recourse to the underlying measure space $\mathcal{M}$. In particular, if two measure spaces have isomorphic measure algebras then the statements of \mbox{Theorem \ref{MainTheorem}} on these two spaces are equivalent.
\begin{remark}\label{MeasSpClass} It is known in measure theory that every finite separable diffuse measure space $\mathcal{M}$ satisfies
$$
\mathcal{B}(\mathcal{M})\simeq\mathcal{B}\left([0,a]\right),
$$
where $a>0$ is a positive real number.
\end{remark}
Thus, without loss of generality, we can restrict ourselves to measure spaces $\mathcal{M}=[0,a]$. The next reduction comes from the following observation.
\begin{remark} If Theorem \ref{MainTheorem*} is true for the finite separable measure space $(\mathcal{M},\mu)$ then it is true also for $(\mathcal{M},\lambda\mu)$ for every $\lambda>0$.
\end{remark}
Indeed, for every $p\in[1,\infty]$ the operator $\operatorname{T}_pf\doteq\lambda^{-\frac1p}f$ defines an isometric isomorphism \newline ${\operatorname{T}_p:L^p(\mathcal{M},\mu)\to L^p(\mathcal{M},\lambda\mu)}$. It is now straightforward to check that if the statements of Theorem \ref{MainTheorem*} hold on $(\mathcal{M},\mu)$ with data $\{\varphi_n\}_{n=1}^\infty$, $\epsilon$, $\delta$, $\chi_E$, $f$, $g$, then they hold on $(\mathcal{M},\lambda\mu)$ with data $\{\operatorname{T}_2\varphi_n\}_{n=1}^\infty$, $\epsilon$, $\lambda\delta$, $\operatorname{T}_\infty\chi_E$, $\operatorname{T}_1f$, $\operatorname{T}_1g$.

Thus we established that without loss of generality we are allowed to prove the theorem just for the unit interval $\mathcal{M}=[0,1]$. In fact, in the next sections we will prove Theorem \ref{MainTheorem} on separable cylindric probability spaces, i.e., separable probability spaces of the form $\mathcal{M}=[0,1]\otimes\mathcal{N}$,  where $\mathcal{N}$ is another probability space. The unit interval is trivially cylindric, $[0,1]\simeq[0,1]\otimes\{1\}$, and it may seem an unnecessary effort to prove the theorem for a cylindric space instead of $[0,1]$. But note that the result cited in Remark \ref{MeasSpClass} is very abstract and the produced isomorphisms are in general far from being geometrically natural. Our proof of Theorem \ref{MainTheorem} is constructive, and the construction of the set $E$ highly depends on the cylindric structure. If the space at hand has a natural cylindric structure then this approach gives a geometrically more sensible set $E$ than what we would expect had we identified the cylinder with the unit interval through a wild measure algebra isomorphism.

\section*{The particular case}

In this section we will prove the main theorem for the particular case where $\mathcal{M}=(\mathcal{M},\Sigma,\mu)$ is a separable cylindric probability space
\begin{equation}\label{MCylindric}
\mathcal{M}=[0,1]\otimes\mathcal{N}.
\end{equation}
Here $\mathcal{N}=(\mathcal{N},\Sigma_0,\nu)$ is any separable probability space. We will write $\mathcal{M}\ni x=(t,y)\in[0,1]\times\mathcal{N}$.

\subsection*{The core lemmata}

First let us state a variant of F\'ej\'er's lemma.
\begin{lemma} Let $a,b\in\mathbb{R}$, $a<b$. For every $f\in L^1[a,b]$ and $g\in L^\infty(\mathbb{R})$, $g$ being $(b-a)$-periodic, $$
\lim_{\lambda\to+\infty}\int_a^bf(t)g(\lambda t)dt=\frac1{b-a}\int_a^bf(t)dt\int_a^bg(t)dt.
$$
\end{lemma}
This lemma is given in \cite[page 77]{Bar64} with $[a,b]=[-\pi,\pi]$, but the proof for arbitrary $a$ and $b$ follows with only trivial modifications.

We proceed to our first critical lemma. 
\begin{lemma}\label{Lemma1o} Let $\Delta=[a,b]\times\Delta_0\in\Sigma$ with $[a,b]\subset[0,1]$ and $\Delta_0\in\Sigma_0$,\quad $0\neq\gamma\in\mathbb{R}$,\quad $\epsilon,\delta\in(0,1)$ and $N\in\mathbb{N}$ be given. Then there exists a function $g\in L^\infty(\mathcal{M})$, a measurable set $\Sigma\ni E\subset\Delta$ and a Fourier polynomial of the form
$$
Q(x)=\sum_{n=N}^MY_n(x),\quad N\le M\in\mathbb{N},
$$
such that

1. $|E|>|\Delta|(1-\delta)$,

2. $g(x)=\gamma$ for $x\in E$\quad and\quad $g(x)=0$ for $x\notin\Delta$,

3. $|\gamma||\Delta|<\|g\|_1<2|\gamma||\Delta|$,

4. $\|Q-g\|_1<\epsilon$,

5. and
$$
\max_{N\le m\le M}\left\|\sum_{n=N}^mY_n\right\|_1\le\frac{|\gamma|\sqrt{|\Delta|(1+\delta)}}{\sqrt{\delta}}.
$$
\end{lemma}
\begin{proof} Set
\begin{equation}\label{delta*Def}
\delta_*\doteq\frac\delta{1+\delta}\in\left(0,\frac12\right).
\end{equation}
Define the $1$-periodic function $I:\mathbb{R}\to\mathbb{R}$ by setting
\begin{equation}
I(t)=1-\frac1{\delta_*}\chi_{[0,\delta_*)}(t)=\begin{cases}
1\quad\mbox{if}\quad t\in[\delta_*,1),\\
1-\frac1{\delta_*}\quad\mbox{if}\quad t\in[0,\delta_*)
\end{cases},
\end{equation}
for $t\in[0,1)$ and continuing periodically. Then obviously
\begin{equation}
\int_0^1I(t)dt=0.
\end{equation}
By F\'ej\'er's lemma
\begin{eqnarray}
\lim_{s\to+\infty}\int_\Delta I(st)\varphi_n^*(t,y)dtdy=\lim_{s\to+\infty}\int_a^bI(st)\int_{\Delta_0}\varphi_n^*(t,y)dydt\nonumber\\
=\lim_{s\to+\infty}\int_0^1I(st)\left[\chi_{[a,b]}(t)\int_{\Delta_0}\varphi_n^*(t,y)dy\right]dt=\int_0^1I(t)dt\int_\Delta\varphi_n^*(x)dx=0.
\end{eqnarray}
Choose a natural number $s_0\in\mathbb{N}$ sufficiently large so that
\begin{equation}\label{s0def}
s_0>\frac{(1-\delta_*)^2}{\delta_*^2(b-a)}\quad\mbox{and}\quad\left|\int_\Delta I(s_0t)\varphi_n^*(t,y)dtdy\right|<\frac{\epsilon}{2N|\gamma|},\quad1\le n\le N.
\end{equation}
Set
\begin{equation}\label{g_E_def}
g(x)\doteq\gamma I(s_0t)\chi_{\Delta}(x),\quad E\doteq\left\{x\in\Delta\,\vline\quad g(x)=\gamma\right\}.
\end{equation}
Then it can be seen that
\begin{equation}\label{mesEineq}
|E|\ge|\Delta|\frac{\lfloor s_0(b-a)\rfloor(1-\delta_*)}{s_0(b-a)}>|\Delta|(1-\delta_*)\left(1-\frac1{s_0(b-a)}\right)>|\Delta|(1-\delta),
\end{equation}
where the first inequality of formula (\ref{s0def}) and then formula (\ref{delta*Def}) were used in the last step. Clearly, $g\in L^\infty(\mathcal{M})$ and thus we have proven Statements 1 and 2. Next we note using (\ref{delta*Def}) that
\begin{equation}
\int_\Delta|I(s_0t)|dx=\int_Edx+\int_{\Delta\setminus E}\left|1-\frac1{\delta_*}\right|dx=|E|+\frac1{\delta}(|\Delta|-|E|),
\end{equation}
and then by $(1-\delta)|\Delta|<|E|<|\Delta|$ we find that
\begin{equation}
|\Delta|<|E|+\frac1{\delta}(|\Delta|-|E|)<2|\Delta|,
\end{equation}
which together entail
\begin{equation}\label{Int|I|dxIneq}
|\Delta|<\int_\Delta|I(s_0t)|dx<2|\Delta|.
\end{equation}
Similarly,
\begin{equation}
\int_\Delta|I(s_0t)|^2dx=\int_Edx+\int_{\Delta\setminus E}\left(1-\frac1{\delta_*}\right)^2dx=|E|+\frac1{\delta^2}(|\Delta|-|E|)<\left(1+\frac1{\delta}\right)|\Delta|.
\end{equation}
Formulae (\ref{g_E_def}) and (\ref{Int|I|dxIneq}) imply that
\begin{equation}
|\gamma||\Delta|<\|g\|_1=\int_\mathcal{M}|g(x)|dx=|\gamma|\int_\Delta|I(s_0t)|dx<2|\gamma||\Delta|,
\end{equation}
which proves Statement 3. In a similar fashion we obtain
\begin{equation}\label{|g|2Ineq}
\|g\|_2^2=\int_\mathcal{M}|g(x)|^2dx=\gamma^2\int_\Delta|I(s_0t)|^2dx<\left(1+\frac1{\delta}\right)\gamma^2|\Delta|.
\end{equation}
We have $g\in L^\infty(\mathcal{M})\subset L^2(\mathcal{M})$, and therefore the Fourier series $\sum Y_n(g)$ converges to $g$ in $L^2(\mathcal{M})$.
Thus we can choose the natural number $M\in\mathbb{N}$ so large that
\begin{equation}\label{S_g2Ineq}
\left\|\sum_{n=1}^M Y_n(g)-g\right\|_2<\frac\epsilon2.
\end{equation}
Further, from formula (\ref{s0def}) we estimate the magnitude of the first $N$ Fourier coefficients of $g$ as
\begin{equation}\label{|cn|Ineq}
|c_n(g)|=\left|\int_\mathcal{M}g(x)\varphi_n^*(x)dx\right|=|\gamma|\left|\int_\Delta I(s_0t)\varphi_n^*(x)dx\right|<\frac\epsilon{2N},\quad1\le n\le N.
\end{equation}
Finally, set
\begin{equation}
Q(x)\doteq\sum_{n=N}^MY_n(x;g),\quad\forall x\in\mathcal{M}.
\end{equation}
In order to prove Statement 4 we write
\begin{eqnarray}
\left\|Q-g\right\|_1\le\left\|Q-g\right\|_2=\left\|\sum_{n=N}^MY_n(g)-g\right\|_2\le\left\|\sum_{n=1}^MY_n(g)-g\right\|_2+\left\|\sum_{n=1}^{N-1}Y_n(g)\right\|_2\nonumber\\
\le\left\|\sum_{n=1}^MY_n(g)-g\right\|_2+\sum_{n=1}^{N-1}|c_n(g)|\|\varphi_n\|_2<\epsilon,
\end{eqnarray}
where formulae (\ref{S_g2Ineq}) and (\ref{|cn|Ineq}) were used in the last step along with the normalization $\|\varphi_n\|_2=1$. Using the pairwise orthogonality of the Fourier components $Y_n(g)$ and formula (\ref{|g|2Ineq}) we can obtain the coarse estimate
\begin{equation}
\left\|\sum_{n=N}^mY_n(g)\right\|_2^2=\sum_{n=N}^m\left\|Y_n(g)\right\|_2^2\le\sum_{n=1}^\infty\left\|Y_n(g)\right\|_2^2=\|g\|_2^2<\left(1+\frac1{\delta}\right)\gamma^2|\Delta|,
\end{equation}
which immediately yields
\begin{equation}
\left\|\sum_{n=N}^mY_n(g)\right\|_1\le\left\|\sum_{n=N}^mY_n(g)\right\|_2<\frac{|\gamma|\sqrt{|\Delta|(1+\delta)}}{\sqrt{\delta}},\quad m>N,
\end{equation}
thus proving Statement 5. Note that  the inequality $\|.\|_1\le\|.\|_2$ used above hold thanks to the convenient assumption that we are in a probability space. $\Box$
\end{proof}

\begin{lemma}\label{Lemma1} Let $f\in L^1(\mathcal{M})$ with $\|f\|_1>0$, $\epsilon,\delta\in(0,1)$, $N_0\in\mathbb{N}$. Then $\exists E\in\Sigma$, $g\in L^\infty(\mathcal{M})$ and
$$
Q(x)=\sum_{n=N_0}^NY_n(x),\quad N\in\mathbb{N},
$$
such that

1. $|E|>1-\delta$,

2. $x\in E$ implies $g(x)=f(x)$,

3. $\frac13\|f\|_1<\|g\|_1<3\|f\|_1$,

4. $\|g-Q\|_1<\epsilon$,

5. and
$$
\sup_{N_0\le m\le N}\left\|\sum_{n=N_0}^mY_n\right\|_1<3\|f\|_1.
$$
\end{lemma}
\begin{proof}For every measurable partition of $\mathcal{N}$
\begin{equation}
\mathcal{N}=\bigsqcup_{i=1}^{\tilde\nu_0}\tilde\Delta_i,\quad\tilde\Delta_i\in\Sigma_0,\quad i\neq j\quad\Rightarrow\quad\left|\tilde\Delta_i\cap\tilde\Delta_j\right|=0,\quad\forall i,j=1,\ldots\tilde\nu_0,\quad\tilde\nu_0\in\mathbb{N},
\end{equation}
and every partition $0=x_0<x_1<\ldots<x_{\bar\nu_0}=1$, $\bar\nu_0\in\mathbb{N}$, of the unit interval, the product partition
\begin{equation}
\Delta_k=[x_{l-1},x_l]\times\tilde\Delta_i,\quad k=\bar\nu_0\cdot i+l-1=1,\ldots,\nu_0\doteq\bar\nu_0\cdot\tilde\nu_0,\quad l=1,\ldots,\bar\nu_0,\quad i=1,\ldots,\tilde\nu_0,
\end{equation}
is a measurable partition of $\mathcal{M}$ with the property that
\begin{equation}\label{PartitionDominance}
\max_{1\le k\le\nu_0}|\Delta_k|\le\max_{1\le l\le\bar\nu_0}|x_l-x_{l-1}|.
\end{equation}
For every product partition $\{\Delta_k\}_{k=1}^{\nu_0}$ as above and every tuple of real numbers $\{\gamma_k\}_{k=1}^{\nu_0}$ consider the step function
\begin{equation}\label{StepFuncDef}
\Lambda(x)=\sum_{k=1}^{\nu_0}\gamma_k\chi_{\Delta_k}(x),\quad\forall x\in\mathcal{M}.
\end{equation}
By the assumption of separability of $\mathcal{M}$ we know that step functions of the form (\ref{StepFuncDef}) subordinate to product partitions are dense in all spaces $L^p(\mathcal{M})$ for $1\le p<\infty$. Choose a product partition and a subordinate step function such that
\begin{equation}\label{|Lambda-f|}
\|\Lambda-f\|_1<\min\left\{\frac12\epsilon,\frac13\|f\|_1\right\}.
\end{equation}
Note that the numbers $\gamma_k$ are not assumed to be distinct, thus we can refine the given partition without changing $\gamma_k$ and the function $\Lambda(x)$. We use the property (\ref{PartitionDominance}) to refine the product partition $\{\Delta_k\}_{k=1}^{\nu_0}$ until it satisfies
\begin{equation}\label{FinePartIneq}
144\gamma_k^2|\Delta_k|(1+\delta)<\delta\|f\|_1^2,\quad k=1,\ldots,\nu_0.
\end{equation}
Now we apply Lemma \ref{Lemma1o} iteratively with
\begin{equation}
\Delta\leftarrow\Delta_k,\quad\gamma\leftarrow\gamma_k,\quad\epsilon\leftarrow\frac1{2^{\nu_0+2}}\min\left\{\epsilon,\|f\|_1\right\},\quad\delta\leftarrow\delta,\quad N\leftarrow N_{k-1}
\end{equation}
for $k=1,\dots,\nu_0$, obtaining at each $k$ a function $g_k\in L^\infty(\mathcal{M})$, a set $\Sigma\ni E_k\subset\Delta_k$, a number $N_{k-1}\le N_k\in\mathbb{N}$ and a Fourier polynomial
\begin{equation}\label{QkDef}
Q_k(x)=\sum_{n=N_{k-1}}^{N_k-1}Y_n(x)
\end{equation}
with the following properties:

$1^\circ$. $|E_k|>|\Delta_k|(1-\delta)$,

$2^\circ$. $g_k(x)=\gamma_k$ for $x\in E_k$ and $g_k(x)=0$ for $x\notin\Delta_k$,

$3^\circ$. $|\gamma_k||\Delta_k|<\|g_k\|_1<2|\gamma_k||\Delta_k|$,

$4^\circ$. $\|Q_k-g_k\|_1<\frac1{2^{\nu_0+2}}\min\{\epsilon,\|f\|_1\}$,

$5^\circ$. and
$$
\max_{N_{k-1}\le m<N_k}\left\|\sum_{n=N_{k-1}}^mY_n\right\|_1\le\frac{|\gamma_k|\sqrt{|\Delta_k|(1+\delta)}}{\sqrt{\delta}}.
$$
Set
\begin{eqnarray}
E\doteq\bigcup_{k=1}^{\nu_0}E_k,\quad g(x)\doteq f(x)-\left[\Lambda(x)-\sum_{k=1}^{\nu_0}g_k(x)\right],\label{EDefgDef}\\
N\doteq N_{\nu_0}-1,\quad Q(x)\doteq\sum_{k=1}^{\nu_0}Q_k(x)=\sum_{n=N_0}^NY_n(x),\quad\forall x\in\mathcal{M}.\label{QDef}
\end{eqnarray}
First we check that from (\ref{StepFuncDef}), ($1^\circ$), ($2^\circ$) and (\ref{EDefgDef}) it follows that
\begin{eqnarray}
|E|=\sum_{k=1}^{\nu_0}|E_k|>\sum_{k=1}^{\nu_0}|\Delta_k|\left(1-\delta\right)=1-\delta,\\
x\in E\quad\Longrightarrow\quad
x\in E_k\quad\Longrightarrow\quad\Lambda(x)=\gamma_k=g_k(x),\,g_l(x)=0,\,l\neq k\quad\Longrightarrow\quad g(x)=f(x),
\end{eqnarray}
so that Statements 1 and 2 are proven. Next we observe using ($4^\circ$), (\ref{|Lambda-f|}), (\ref{EDefgDef}) and (\ref{QDef}) that
\begin{equation}
\|Q-g\|_1=\left\|\sum_{k=1}^{\nu_0}\left[Q_k-g_k\right]+\left[f-\Lambda\right]\right\|_1\le\sum_{k=1}^{\nu_0}\|Q_k-g_k\|_1+\|f-\Lambda\|_1<\epsilon,
\end{equation}
which proves Statement 4. Further, from (\ref{StepFuncDef}), (\ref{|Lambda-f|}), ($2^\circ$) and ($3^\circ$) we deduce that
\begin{eqnarray}
\|g\|_1\le\sum_{k=1}^{\nu_0}\|g_k\|_1+\|f-\Lambda\|_1\le2\sum_{k=1}^{\nu_0}|\gamma_k||\Delta_k|+\|f-\Lambda\|_1\nonumber\\
=2\|\Lambda\|_1+\|f-\Lambda\|_1\le3\|f-\Lambda\|_1+2\|f\|_1<3\|f\|_1.
\end{eqnarray}
Moreover, the same formulae also imply
\begin{eqnarray}
\|g\|_1+\frac13\|f\|_1>\|g\|_1+\|f-\Lambda\|_1\ge\|g-f+\Lambda\|_1\nonumber\\
=\sum_{k=1}^{\nu_0^2}\|g_k\|_1>\sum_{k=1}^{\nu_0^2}|\gamma_k||\Delta_k|=\|\Lambda\|_1\ge\bigl|\|\Lambda-f\|_1-\|f\|_1\bigr|>\frac23\|f\|_1,
\end{eqnarray}
i.e., $\|f\|_1<3\|g\|_1$, thus proving Statement 3. In order to prove Statement 5 let us fix an $N_0\le m\le N$. Then there is a $1\le k_0\le\nu_0$ such that $N_{k_0-1}\le m<N_{k_0}$, and thus by (\ref{QkDef}) and (\ref{QDef}) we have
\begin{equation}
\sum_{n=N_0}^mY_n(x)=\sum_{k=1}^{k_0-1}Q_k(x)+\sum_{N_{k_0-1}}^mY_n(x).
\end{equation}
Finally we use this along with formulae ($3^\circ$), ($4^\circ$), ($5^\circ$) and (\ref{FinePartIneq}) to obtain
\begin{eqnarray}
\left\|\sum_{n=N_0}^mY_n(x)\right\|_1\le\sum_{k=1}^{k_0-1}\left\|Q_k-g_k\right\|_1+\sum_{k=1}^{k_0-1}\|g_k\|_1+\left\|\sum_{k=N_{k_0-1}}^mY_n\right\|_1\nonumber\\
<\frac14\|f\|_1+2\|\Lambda\|_1+\frac{|\gamma_{k_0}|\sqrt{|\Delta_{k_0}|(1+\delta)}}{\sqrt{\delta}}<3\|f\|_1,
\end{eqnarray}
and this completes the proof. $\Box$
\end{proof}

\begin{lemma}\label{PolynomSeriesLemma} Let $\{R_k\}_{k=1}^\infty$ be any fixed ordering of the set of all nonzero Fourier polynomials with rational coefficients into a sequence. Then for every $f\in L^1(\mathcal{M})$ and sequence $\{b_s\}_{s=1}^\infty$ of positive numbers $b_s>0$ there exists subsequence $\{R_{k_s}\}_{s=0}^\infty$ such that

1. $\|R_{k_0}-f\|_1\le\frac12\|f\|_1$

2. $\|R_{k_s}\|_1<b_s$ for $s\ge1$

3. $\sum_{s=0}^\infty R_{k_s}=f$ in $L^1(\mathcal{M})$.
\end{lemma}
\begin{proof} Let us first convince ourselves that Fourier polynomials with rational coefficients are dense in $L^1(\mathcal{M})$. Indeed, by the assumption of separability, step functions are dense in $L^1(\mathcal{M})$, but they all belong also to the separable Hilbert space $L^2(\mathcal{M})$. On the other hand, Fourier polynomials are clearly dense in $L^2(\mathcal{M})$. And finally, an arbitrary Fourier polynomial can be approximated in $L^2(\mathcal{M})$ by a Fourier polynomial with rational coefficients (this amounts to approximating the Fourier coefficients by rational numbers). A three-epsilon argument together with $\|.\|_1\le\|.\|_2$ then yields the assertion.

Using the denseness of $\{R_k\}_{k=1}^\infty$ let us choose a natural number $k_0\in\mathbb{N}$ such that
\begin{equation}
\|R_{k_0}-f\|_1\le\frac12\min\left\{\|f\|_1, b_1\right\}.
\end{equation}
Then we can choose further natural numbers $k_s\in\mathbb{N}$ iteratively as follows. For every $s\in\mathbb{N}$, again by using the denseness argument, choose a number $k_s$ so that $k_s>k_{s-1}$ and
\begin{equation}
\left\|f-\sum_{r=0}^sR_{k_r}\right\|_1=\left\|R_{k_s}-\left(f-\sum_{r=0}^{s-1}R_{k_r}\right)\right\|_1<\frac12\min\left\{b_s, b_{s+1},\frac1s\right\},\quad\forall s\in\mathbb{N}.
\end{equation}
Statements 1 and 3 are clearly satisfied. For Statement 2 we have
\begin{eqnarray}
\left\|R_{k_s}\right\|_1=\left\|R_{k_s}-\left(f-\sum_{r=0}^{s-1}R_{k_r}\right)+\left(f-\sum_{r=0}^{s-1}R_{k_r}\right)\right\|_1\nonumber\\
\le\left\|R_{k_s}-\left(f-\sum_{r=0}^{s-1}R_{k_r}\right)\right\|_1+\left\|f-\sum_{r=0}^{s-1}R_{k_r}\right\|_1<\frac12b_s+\frac12b_s=b_s,\quad\forall s\in\mathbb{N}.
\end{eqnarray}
Lemma is proven. $\Box$
\end{proof}

\subsection*{The main theorem}

Here we will prove Theorem \ref{MainTheorem} for the particular case of $(\mathcal{M},\Sigma,\mu)$ being a separable cylindric probability space as in (\ref{MCylindric}).

\begin{proof} Recall that $\epsilon,\delta>0$ and $f\in L^1(\mathcal{M})$ with $\|f\|_1>0$ are given. Denote
\begin{equation}\label{epsilon0}
\epsilon_0\doteq\min\left\{\epsilon,\|f\|_1\right\}.
\end{equation}
Let $\{R_k\}_{k=1}^\infty$ be any ordering of the set of all nonzero Fourier polynomials with rational coefficients into a sequence. Iteratively applying Lemma \ref{Lemma1} with
\begin{equation}
f\leftarrow R_k,\quad\epsilon\leftarrow\frac{\epsilon_0}{2^{k+7}},\quad\delta\leftarrow\frac\delta{2^k},\quad N_0\leftarrow N_{k-1}
\end{equation}
for $k=1,2,\dots$ we obtain for each $k\in\mathbb{N}$ a subset $\tilde E_k\in\Sigma$, a function $\tilde g_k\in L^\infty(\mathcal{M})$, a number $N_{k-1}\le N_k\in\mathbb{N}$ (set $N_0=0$) and a Fourier polynomial
\begin{equation}
\tilde Q_k(x)=\sum_{n=N_{k-1}}^{N_k-1}\tilde Y_n(x)
\end{equation}
with the following properties:

$1^\dagger$. $|\tilde E_k|>1-\frac\delta{2^k}$,

$2^\dagger$. $x\in\tilde E_k$ implies $\tilde g_k(x)=R_k(x)$,

$3^\dagger$. $\frac13\|R_k\|_1<\|\tilde g_k\|_1<3\|R_k\|_1$,

$4^\dagger$. $\|\tilde g_k-\tilde Q_k\|_1<\epsilon_02^{-k-7}$,

$5^\dagger$. and
$$
\sup_{N_{k-1}\le m<N_k}\left\|\sum_{n=N_{k-1}}^m\tilde Y_n\right\|_1<3\|R_k\|_1.
$$
Define the desired set $E$ as
\begin{equation}
E\doteq\bigcap_{k=1}^\infty\tilde E_k.
\end{equation}
Observe from ($1^\dagger$) that
\begin{equation}
|E|=1-|\mathcal{M}\setminus E|\ge1-\sum_{s=1}^\infty|\mathcal{M}\setminus\tilde E_k|>1-\sum_{k=1}^\infty\frac\delta{2^k}=1-\delta.
\end{equation}
Note that $E$ is universal, i.e., independent of $f$.

Let $\{R_{k_s}\}_{s=0}^\infty$ be the subsequence of Fourier polynomials provided by Lemma \ref{PolynomSeriesLemma} applied with
\begin{equation}
f\leftarrow f,\quad b_s\leftarrow\frac{\epsilon_0}{2^{s+6}}.
\end{equation}
It satisfies

$1^\circ$. $\|R_{k_0}-f\|_1\le\frac12\|f\|_1$,

$2^\circ$. $\|R_{k_s}\|_1<\epsilon_02^{-s-6}$ for $s\ge1$,

$3^\circ$. $\sum_{s=0}^\infty R_{k_s}=f$ in $L^1(\mathcal{M})$.

\noindent
We want to use mathematical induction in order to define a sequence of natural numbers \newline${1<\nu_1<\nu_2<\ldots}$ and a sequence of functions $\{g_s\}_{s=1}^\infty$, $g_s\in L^1(\mathcal{M})$, such that for all $s\in\mathbb{N}$ we have

$1^*.$ $x\in\tilde E_{\nu_s}$ implies $g_s(x)=R_{k_s}(x)$,

$2^*.$ $\|g_s\|_1<\epsilon_02^{-s-2}$,

$3^*.$
$$
\left\|\sum_{j=1}^s[\tilde Q_{\nu_j}-g_j]\right\|_1<\frac{\epsilon_0}{2^{s+6}},
$$

$4^*.$
$$
\max_{N_{\nu_s-1}\le m<N_{\nu_s}}\left\|\sum_{n=N_{\nu_s-1}}^m\tilde Y_n\right\|_1<\frac{\epsilon_0}{2^s}.
$$
Assume that for some $s\in\mathbb{N}$, the choice of $1<\nu_1<\nu_2<\ldots<\nu_{s-1}$ and $g_1, g_2,\ldots,g_{s-1}$ satisfying ($3^*$) has been already made (for $s=1$ this is trivially correct). Remember that by $\sigma(h)$  we have denoted the $\{\varphi_n\}$-spectrum of a function $h\in L^1(\mathcal{M})$, i.e., the support of its Fourier series. Using the denseness of $\{R_k\}_{k=1}^\infty$ (see Lemma \ref{PolynomSeriesLemma}) choose a natural number $\nu_s\in\mathbb{N}$ such that $N_{\nu_1-1}>\max\sigma(R_{k_0})$ and $\nu_s>\nu_{s-1}$ for $s>1$, and
\begin{equation}\label{nuqDef}
\left\|R_{\nu_s}-\left(R_{k_s}-\sum_{j=1}^{s-1}[\tilde Q_{\nu_j}-g_j]\right)\right\|_1<\frac{\epsilon_0}{2^{s+7}}.
\end{equation}
Then by ($2^\circ$) and ($3^*$) we have for all $s\in\mathbb{N}$ that
\begin{equation}
\left\|R_{k_s}-\sum_{j=1}^{s-1}[\tilde Q_{\nu_j}-g_j]\right\|_1\le\|R_{k_s}\|_1+\left\|\sum_{j=1}^{s-1}[\tilde Q_{\nu_j}-g_j]\right\|_1<\frac{3\,\epsilon_0}{2^{s+6}},
\end{equation}
which combined with (\ref{nuqDef}) implies
\begin{equation}\label{RnusEst}
\|R_{\nu_s}\|_1\le\left\|R_{\nu_s}-\left(R_{k_s}-\sum_{j=1}^{s-1}[\tilde Q_{\nu_j}-g_j]\right)\right\|_1+\left\|R_{k_s}-\sum_{j=1}^{s-1}[\tilde Q_{\nu_j}-g_j]\right\|_1<\frac{7\,\epsilon_0}{2^{s+7}}.
\end{equation}
Set
\begin{equation}
g_s(x)\doteq R_{k_s}(x)+\tilde g_{\nu_s}(x)-R_{\nu_s}(x).
\end{equation}
Condition ($1^*$) is easily satisfied thanks to ($2^\dagger$) with $k=\nu_s$. For condition ($2^*$) we write
\begin{eqnarray}\label{|gs|Est}
\|g_s\|_1=\|R_{k_s}+\tilde g_{\nu_s}-R_{\nu_s}\|_1\nonumber\\
\le\left\|R_{\nu_s}-R_{k_s}+\sum_{j=1}^{s-1}[\tilde Q_{\nu_j}-g_j]\right\|_1+\left\|\sum_{j=1}^{s-1}[\tilde Q_{\nu_j}-g_j]\right\|_1+\|\tilde g_{\nu_s}\|_1<\frac{\epsilon_0}{2^{s+2}},
\end{eqnarray}
where we used (\ref{nuqDef}), ($3^\dagger$), ($3^*|_{s-1}$) and (\ref{RnusEst}) in the last step. To show that condition ($3^*$) is satisfied we observe that
\begin{eqnarray}
\left\|\sum_{j=1}^s[\tilde Q_{\nu_j}-g_j]\right\|_1=\left\|\tilde Q_{\nu_s}-g_s+\sum_{j=1}^{s-1}[\tilde Q_{\nu_j}-g_j]\right\|_1=\left\|\tilde Q_{\nu_s}-R_{k_s}-\tilde g_{\nu_s}+R_{\nu_s}+\sum_{j=1}^{s-1}[\tilde Q_{\nu_j}-g_j]\right\|_1\nonumber\\
\le\left\|\tilde Q_{\nu_s}-\tilde g_{\nu_s}\right\|_1+\left\|R_{\nu_s}-\left(R_{k_s}-\sum_{j=1}^{s-1}[\tilde Q_{\nu_j}-g_j]\right)\right\|_1<\frac{\epsilon_0}{2^{s+6}},
\end{eqnarray}
where ($4^\dagger$), (\ref{nuqDef}) and $\nu_s>s$ were used in the second inequality. Finally we satisfy condition ($4^*$) using ($5^\dagger$) and (\ref{RnusEst}),
\begin{equation}
\max_{N_{\nu_s-1}\le m<N_{\nu_s}}\left\|\sum_{n=N_{\nu_s-1}}^m\tilde Y_n\right\|_1<3\|R_{\nu_s}\|_1<\frac{\epsilon_0}{2^s}.
\end{equation}
The iteration is thus complete, and by mathematical induction we construct the sequences $\{\nu_s\}_{s=1}^\infty$ and $\{g_s\}_{s=1}^\infty$ satisfying conditions ($1^*$) through ($4^*$) for all $s\in\mathbb{N}$.
Define
\begin{equation}\label{gDef}
g(x)\doteq R_{k_0}(x)+\sum_{s=1}^\infty g_s(x),\quad\forall x\in\mathcal{M}.
\end{equation}
From (\ref{|gs|Est}) it follows that
\begin{equation}\label{Sum|gs|Est}
\sum_{s=1}^\infty\|g_s\|_1<\frac{13\,\epsilon_0}{64}<\infty,
\end{equation}
thus $g\in L^1(\mathcal{M})$. The construction is now complete, and it remains to verify the statements of the theorem.

To prove Statement 2 of the theorem we note that $x\in E$ means $x\in \tilde E_{\nu_s}$, and hence by ($2^\dagger$) $g_s(x)=R_{k_s}(x)$ for all $s\in\mathbb{N}$. It then follows from ($3^\circ$) that
\begin{equation}
g(x)=R_{k_0}(x)+\sum_{s=1}^\infty g_s(x)=\sum_{s=0}^\infty R_{k_s}(x)=f(x),\quad\forall x\in E.
\end{equation}
Let $\{Y_n\}_{n=1}^\infty$ be the series of $Y_n=c_n\varphi_n$ such that
\begin{equation}
\sum_{n=1}^{N_{\nu_s}-1}Y_n=R_{k_0}+\sum_{j=1}^s\tilde Q_{\nu_j},\quad\forall s\in\mathbb{N}.
\end{equation}
Let $m\in\mathbb{N}$, and let $r\in\mathbb{N}$ be the largest natural number such that $N_{\nu_{r}-1}\le m$ (if $m<N_{\nu_1-1}$ set $r=1$). Set ${m_*\doteq\min\{m,N_{\nu_{r}}-1\}}$. Then by (\ref{gDef}), ($3^*$), ($4^*$) and (\ref{|gs|Est}) we get
\begin{eqnarray}
\left\|\sum_{n=1}^m Y_n-g\right\|_1=\left\|\sum_{n=1}^{m_*} Y_n-g\right\|_1=\left\|\sum_{j=1}^{r-1}\tilde Q_{\nu_j}+\sum_{n=N_{\nu_{r}-1}}^{m_*}\tilde Y_n-\sum_{j=1}^{r-1}g_j-\sum_{j=r}^\infty g_j\right\|_1\nonumber\\
\le\left\|\sum_{j=1}^{r-1}\left[\tilde Q_{\nu_j}-g_j\right]\right\|_1+\left\|\sum_{n=N_{\nu_{r}-1}}^{m_*}\tilde Y_n\right\|_1+\left\|\sum_{j=r}^\infty g_j\right\|_1<\frac{23\,\epsilon_0}{2^{r+4}}.\label{EstFourRem}
\end{eqnarray}
Now as $m\to\infty$ obviously $r\to\infty$ as well, thus making the above expression vanish, which proves that $\sum Y_n$ is the Fourier series of $g$, i.e., $Y_n=Y_n(g)$, and it converges to $g$ as required in Statement 3. Further, from ($2^\circ$), ($3^\circ$), (\ref{gDef}) and (\ref{Sum|gs|Est}) we have that
\begin{equation}\label{|f-g|Est}
\left\|f-g\right\|_1=\left\|\sum_{s=1}^\infty R_{k_s}-\sum_{s=1}^\infty g_s\right\|_1\le\sum_{s=1}^\infty \|R_{k_s}\|_1+\sum_{s=1}^\infty\|g_s\|_1<\frac{7\,\epsilon_0}{32},
\end{equation}
which in view of (\ref{epsilon0}) proves Statement 1. Finally, using (\ref{EstFourRem}) and (\ref{|f-g|Est}) we establish that
\begin{equation}
\left\|\sum_{n=1}^m Y_n\right\|_1\le\left\|\sum_{n=1}^m Y_n-g\right\|_1+\left\|f-g\right\|_1+\left\|f\right\|_1<\frac{15\,\epsilon_0}{16}+\|f\|_1<2\|f\|_1,
\end{equation}
but also
\begin{equation}\label{PartialSumIntermed}
\left\|\sum_{n=1}^m Y_n\right\|_1\le\left\|\sum_{n=1}^m Y_n-g\right\|_1+\left\|g\right\|_1<\frac{23}{32}\|f\|_1+\|g\|_1.
\end{equation}
Note that by (\ref{|f-g|Est})
\begin{equation}
\|f\|_1\le\|g\|_1+\|f-g\|_1<\|g\|_1+\frac7{32}\|f\|_1,
\end{equation}
and therefore $25\|f\|_1<32\|g\|_1$. This together with (\ref{PartialSumIntermed}) yields
\begin{equation}
\left\|\sum_{n=1}^m Y_n\right\|_1<\frac{23}{25}\|g\|_1+\|g\|_1<2\|g\|_1,
\end{equation}
which establishes Statement 4. The proof of the theorem is accomplished. $\Box$
\end{proof}

\section*{Compact groups}

Let $G$ be a compact Hausdorff topological group, $\Sigma$ the Borel $\sigma$-algebra and $d\mu(x)=dx$ the normalized Haar measure. Then $\mu$ is diffuse if and only if $G$ is infinite, which we will assume here. Now by \cite[Theorem 28.2]{HeRo70} we have that $\dim L^2(G)=\operatorname{w}(G)$, therefore $(G,\Sigma,\mu)$ is separable if and only if $G$ is second countable. This will also be assumed in what follows. This implies in particular, through Peter-Weyl theorem, that the dual $\hat G$ is countable. For a detailed exposition of harmonic analysis on compact groups consult, e.g., \cite{RuTu09} or \cite{Fol15}.

For every irreducible unitary representation $\rho\in\hat G$ (or rather $[\rho]\in\hat G$), let $\mathcal{H}_\rho$ be the representation Hilbert space of dimension $\dim\mathcal{H}_\rho\doteq d_\rho\in\mathbb{N}$. Choose an arbitrary orthonormal basis $\{e^\rho_i\}_{i=1}^{d_\rho}$ of $\mathcal{H}_\rho$, and denote
\begin{equation}
\varphi_{\rho,i,j}(x)\doteq\sqrt{d_\rho}(\rho(x)e^\rho_j, e^\rho_i),\quad\forall x\in G,\quad i,j=1,\dots,d_\rho,\quad\forall\rho\in\hat G.
\end{equation}
By the Peter-Weyl theorem $\{\varphi_{\rho,i,j}\}$ is an orthonormal basis in $L^2(G)$. Moreover, since $\rho(x)$ is unitary for all $x\in G$, we get
\begin{equation}
|\varphi_{\rho,i,j}(x)|=\sqrt{d_\rho}\,|(\rho(x)e^\rho_j,e^\rho_i)|\le\sqrt{d_\rho}\,\|\rho(x)e^\rho_j\|\|e^\rho_i\|\le\sqrt{d_\rho},
\end{equation}
so that $\varphi_{\rho,i,j}\in L^\infty(G)$. Therefore, if we put an arbitrary (total) order on $\hat G$ then Theorem \ref{MainTheorem} is directly applicable to $(G,\Sigma,\mu)$ with $\{\varphi_{\rho,i,j}\}$.

But the arbitrary choice of the bases $\{e^\rho_i\}_{i=1}^{d_\rho}$ is artificial from the viewpoint of the group $G$. The more natural construction is the operator valued Fourier transform,
\begin{equation}
\hat f(\rho)=\int_Gf(x)\rho^*(x)dx,\quad\forall f\in L^1(G),
\end{equation}
and the corresponding block Fourier series
\begin{eqnarray}
\sum_{\rho\in\hat G}d_\rho\tr\left[\hat f(\rho)\rho(x)\right]=\sum_{\rho\in\hat G}\sum_{i,j=1}^{d_\rho}Y_{\rho,i,j}(x,f),\nonumber\\
Y_{\rho,i,j}(x,f)=c_{\rho,i,j}(f)\varphi_{\rho,i,j}(x),\quad c_{\rho,i,j}(f)=(f,\varphi_{\rho,i,j})_2.
\end{eqnarray}
More generally, if we work on a homogeneous space $\mathcal{M}\simeq G/H$ of a compact group $G$ as above with (closed) isotropy subgroup $H\subset G$, then we define multiplicities
\begin{equation}
d^H_\rho=\operatorname{mult}(\id,\rho\,\vline_H),\quad\forall\rho\in\hat G
\end{equation}
(this reduces to $d^H_\rho=d_\rho$ if $H=\{\id\}$ as before). Moreover, we restrict to
\begin{equation}
\widehat{G/H}=\left\{\rho\in\hat G\,\vline\quad d^H_\rho>0\right\}.
\end{equation}
A point $x\in G/H$ is a coset $x=\mathrm{x}H$, $\mathrm{x}\in G$. If $dh$ is the normalized Haar measure on $H$ then there is a unique normalized left $G$-invariant measure $\mu$ on $G/H$ (the pullback of $d\mathrm{x}$ through the quotient map) such that
\begin{equation}
\int_Gf(\mathrm{x})d\mathrm{x}=\int_{G/H}\left(\int_Hf(\mathrm{x}h)dh\right)d\mu(x),\quad\forall f\in C(G).
\end{equation}
Denote
\begin{equation}
\mathbb{P}_H\doteq\int_H\rho(h)dh,\quad\rho(x)\doteq\rho(\mathrm{x})\mathbb{P}_H,\quad\forall\rho\in\widehat{G/H},\quad\forall x=\mathrm{x}H\in G/H.
\end{equation}
Note that $d^H_\rho=\dim\mathbb{P}_H\mathcal{H}_\rho$. Let $\{e^\rho_i\}_{i=1}^{d_\rho}$ be an orthonormal basis in $\mathcal{H}_\rho$ as before, and choose an orthonormal basis $\{e^\rho_\alpha\}_{\alpha=1}^{d^H_\rho}$ in $\mathbb{P}_H\mathcal{H}_\rho$. Now the Fourier transform of a function $f\in L^1(G/H)$ becomes
\begin{equation}
\hat f(\rho)=\int_{G/H}f(x)\rho^*(x)d\mu(x),\quad\forall\rho\in\widehat{G/H},
\end{equation}
and the corresponding block Fourier series is
\begin{eqnarray}\label{BlockFourierSeries}
\sum_{\rho\in\widehat{G/H}}d_\rho\tr\left[\hat f(\rho)\rho(x)\right]=\sum_{\rho\in\widehat{G/H}}\sum_{i=1}^{d_\rho}\sum_{\alpha=1}^{d^H_\rho}Y_{\rho,i,\alpha}(x;f),\\
Y_{\rho,i,\alpha}(x;f)=c_{\rho,i,\alpha}(f)\varphi_{\rho,i,\alpha}(x),\quad c_{\rho,i,\alpha}(f)=(f,\varphi_{\rho,i,\alpha})_2,\quad\varphi_{\rho,i,\alpha}(x)=\sqrt{d_\rho}(\rho(x)e^\rho_\alpha,e^\rho_i).
\end{eqnarray}
If the homogeneous space $G/H$ is infinite then the invariant measure $\mu$ is diffuse. Applying Theorem \ref{MainTheorem} to $(G/H,\Sigma_H,\mu)$ ($\Sigma_H$ is the Borel $\sigma$-algebra on $G/H$) and the system $\{\varphi_{\rho,i,\alpha}\}$ with $\widehat{G/H}$ ordered arbitrarily, we obtain the following modification.
\begin{theorem}\label{MainTheoremGroup} Let $\mathcal{M}=G/H$ be an infinite homogeneous space of a compact second countable Hausdorff group $G$ with closed isotropy subgroup $H\subset G$. For every $\epsilon,\delta>0$ there exists a measurable subset $E\in\Sigma_H$ with measure $|E|>1-\delta$ and with the following property; for each function $f\in L^1(G/H)$ with $\|f\|_1>0$ there exists an approximating function $g\in L^1(G/H)$ that satisfies:

1. $\|f-g\|_1<\epsilon$,

2. $f=g$ on $E$,

3. the block Fourier series (\ref{BlockFourierSeries}) of $g$ converges in $L^1(G/H)$,

4. we have
$$
\sup_{\rho\in\widehat{G/H}}\left\|\sum_{\varrho\le\rho}d_\varrho\tr\left[\hat g(\varrho)\varrho(x)\right]\right\|_1<2\min\left\{\|f\|_1,\|g\|_1\right\}.
$$
\end{theorem}
Note that $E$ depends on the chosen order in $\widehat{G/H}$.

As discussed before, the proof of the above theorem becomes more constructive and transparent if we have a natural cylindric structure in $G/H$. Then the proof based on the cylindric structure becomes directly applicable (without intermediate measurable transformations), and each step in the proof retains its original interpretation in terms of cylindrical coordinates. To this avail, below we will establish a natural Borel almost isomorphism of measure spaces between $G/H$ and the cylindric space $K\times K\backslash G/H$ for certain infinite closed subgroups $K\subset G$, which will provide an obvious identification of all spaces $L^p(G/H)$ and $L^p(K\times K\backslash G/H)$. Note that since $K$ is infinite and compact, the probability space $(K,dk)$ is isomorphic to the unit interval.
\begin{definition} A measurable map between two measure spaces $\varphi:(\Omega_1,\mu_1)\to(\Omega_2,\mu_2)$ is an almost isomorphism of measure spaces if there exist full measure subspaces $X\subset\Omega_1$ and $Y\subset\Omega_2$, $|\Omega_1\setminus X|=|\Omega_2\setminus Y|=0$, such that the restriction of $\varphi$ is an isomorphism of measure spaces $\varphi|_X:(X,\mu_1)\to(Y,\mu_2)$.
\end{definition}

If $K\subset G$ is a closed subgroup of $G$ then denote by $G/H^{(K)}\subset G/H$ the set of those points $x\in G/H$ with a non-trivial stabilizer within $K$,
\begin{equation}
G/H^{(K)}=\left\{x\in G/H\,\vline\quad\exists\,\id\neq k\in K\quad\mbox{s.t.}\quad kx=x\right\}.
\end{equation}
Let us now fix an infinite closed subgroup $K\subset G$, and let $\operatorname{q}:G/H\to K\backslash G/H$ be the natural quotient map. Then the pullback measure $\nu=\mu\circ\operatorname{q}^{-1}=\operatorname{q}^*\mu$ is the natural probability measure on $K\backslash G/H$. Provided that the subset $G/H^{(K)}$ in $G/H$ is $\mu$-null, we obtain a natural product structure in the following way.
\begin{proposition}\label{CylStrCompHomSpProp} If $\left|G/H^{(K)}\right|_\mu=0$ then there exists a Borel almost isomorphism $\varphi:K\times K\backslash G/H\to G/H$ such that
\begin{equation}
\varphi(k'k,Kx)=k'\varphi(k,Kx),\quad
\operatorname{q}\left(\varphi(k,Kx)\right)=Kx,\quad\forall k,k'\in K,\quad\forall Kx\in K\backslash G/H.\label{varphiprops}
\end{equation}
\end{proposition}
\begin{proof} Both $G/H$ and $K\backslash G/H$ are compact Hausdorff second countable, hence metrizable by Urysohn's metrization theorem. The canonical quotient map $\operatorname{q}:G/H\to K\backslash G/H$ is a continuous surjection between compact metrizable spaces. By Federer-Morse theorem there exists a Borel subset $Z\subset G/H$ such that the restriction $\operatorname{q}|_Z:Z\to K\backslash G/H$ is a Borel isomorphism. Let $W\doteq Z\setminus G/H^{(K)}$ and $X=\operatorname{q}(W)$, so that $\operatorname{q}|_W:W\to X$ is a Borel isomorphism. Define $\varphi:K\times X\to G/H$ by setting
\begin{equation}
\varphi(k,Kx)\doteq k\cdot\operatorname{q}|_W^{-1}(Kx),\quad\forall k\in K,\quad\forall Kx\in X.
\end{equation}
$\varphi$ is Borel bi-measurable, since it is the composition of bi-measurable maps $(k,x)\mapsto k\cdot x$ and $\operatorname{q}|_W^{-1}$. The properties (\ref{varphiprops}) are easily implied by the definition of $\varphi$. The map $\varphi$ is also injective. Indeed, if $\varphi(k_1,Kx_1)=\varphi(k_2,Kx_2)$ then
\begin{equation}
\operatorname{q}(\varphi(k_1,Kx_1))=Kx_1=\operatorname{q}(\varphi(k_2,Kx_2))=Kx_2,
\end{equation}
and
\begin{equation}
\varphi(k_1,Kx_1)=k_1\varphi(\id,Kx_1)=\varphi(k_2,Kx_2)=k_2\varphi(\id,Kx_1),
\end{equation}
so that $k_2^{-1}k_1\varphi(\id,Kx_1)=\varphi(\id,Kx_1)$. If $k_1\neq k_2$ then $\varphi(\id,Kx_1)$ has a non-trivial stabilizer, that is, $\varphi(\id,Kx_1)\in G/H^{(K)}$. But $\varphi(\id,Kx_1)=\operatorname{q}|_W^{-1}(Kx_1)\in W$ and $W\cap G/H^{(K)}=\emptyset$, which is a contradiction. Thus $k_1=k_2$ and the injectivity is proven. Denoting $Y\doteq\varphi(X)\subset G/H$ we see that $\varphi:X\to Y$ is a Borel isomorphism.

For every $f\in C(G/H)$, by measure disintegration theorem, we have
\begin{equation}
\int_{G/H}f(x)d\mu(x)=\int_{K\backslash G/H}d\nu(Kx)\int_Kf(kx)dk.
\end{equation}
Let $\chi_X$ and $\chi_Y$ be the indicator functions of the subsets $X$ and $Y$, respectively. Since $K\cdot Y\subset Y$ we have that $\chi_Y(x)=\chi_X(Kx)$. It follows that
\begin{eqnarray}
\int_Yf(x)d\mu(x)=\int_{G/H}f(x)\chi_Y(x)d\mu(x)=\int_{K\backslash G/H}d\nu(Kx)\int_Kf(kx)\chi_Y(kx)dk=\nonumber\\
\int_{K\backslash G/H}\chi_X(Kx)d\nu(Kx)\int_Kf(kx)dk=\int_Xd\nu(Kx)\int_Kf(\varphi(k,Kx))dk,
\end{eqnarray}
which shows that $\varphi:(X,dk\otimes\nu)\to(Y,\mu)$ is a measure space isomorphism.

Finally, let us note that $K\cdot Z=G/H$. Indeed, for every $x\in G/H$ we have that $z=\operatorname{q}|_Z^{-1}(\operatorname{q}(x))\in Z$, and since $\operatorname{q}(x)=\operatorname{q}(z)$ we have that $\exists k\in K$ such that $kz=x$. On the other hand, it is easy to see that the subset $G/H^{(K)}\subset G/H$ is left $K$-invariant, for if $x\in G/H^{(K)}$ with $k_x\in K$ such that $k_xx=x$ then for every $k\in K$ it follows that $k_yy=y$, where $y=kx$ and $k_y=k_0k_xk_0^{-1}$, which means that $y\in G/H^{(K)}$. Therefore
\begin{equation}
\left|G/H\setminus Y\right|_\mu=\left|K\cdot Z\setminus K\cdot W\right|_\mu=\left|K\cdot(Z\setminus W)\right|_\mu\le\left|K\cdot G/H^{(K)}\right|_\mu=\left|G/H^{(K)}\right|_\mu=0,
\end{equation}
so that $|Y|_\mu=|X|_{dk\otimes\nu}=1$. This completes the proof. $\Box$
\end{proof}

\subsection*{Spheres}

As an instructive illustration of the above constructions we will consider spheres $\mathbb{S}^d$, $2\le d\in\mathbb{N}$, with their Euclidean (Lebesgue) probability measures (surface area normalized to one). For $d=2$ the Statements 2 and 3 of Theorem \ref{MainTheorem} were obtained in \cite{Gri90}.

The sphere $\mathbb{S}^d$ can be considered as the homogeneous space $G/H$ with $G=\mathrm{SO}(d+1)$ and $H=\mathrm{SO}(d)$. Harmonic analysis in these homogeneous spaces is a classical subject widely available in the literature (see e.g. \cite{VK93}). The dual space $\widehat{G/H}$ consists of irreducible representations by harmonic polynomials of fixed degree $\rho\in\mathbb{N}_0$, and it is conveniently ordered according to that degree, $\rho\in\widehat{G/H}\simeq\mathbb{N}_0$. The dimension of the representation $\rho$ is
\begin{equation}
d_\rho=\dim\mathcal{H}_\rho={{d+\rho}\choose{d}}-{{d+\rho-2}\choose{d}},\quad\forall\rho\in\mathbb{N}_0,
\end{equation}
whereas the multiplicities are all $d^H_\rho=1$. We choose standard spherical coordinates $x=(\theta_1,\ldots,\theta_{d-1},\phi)$, where $\theta_j\in[0,\pi]$, $j=1,\ldots,d-1$, and $\phi\in[0,2\pi)$. The orthonormal system $\{\varphi_{\rho,i,\alpha}\}$ in this case consists of spherical harmonics
\begin{equation}
\varphi_{\rho,i,\alpha}(x)=Y_\rho^i(\theta_1,\ldots,\theta_{d-1},\phi),\quad\forall\rho\in\mathbb{N}_0,\quad i=1,\ldots,d_\rho.
\end{equation}
The block Fourier series of a function $f\in L^1(\mathbb{S}^d)$ is
\begin{eqnarray}
\sum_{\rho=1}^\infty\sum_{i=1}^{d_\rho}\hat f(\rho;i)Y_\rho^i(\theta_1,\ldots,\theta_{d-1},\phi),\\
\hat f(\rho;i)=\int_0^{2\pi}\int_0^\pi\ldots\int_0^\pi f(\theta_1,\ldots,\theta_{d-1},\phi)\bar Y_\rho^i(\theta_1,\ldots,\theta_{d-1},\phi)d\mu(\theta_1,\ldots,\theta_{d-1},\phi),\\
d\mu(\theta_1,\ldots,\theta_{d-1},\phi)=\frac{\Gamma(\frac{d+1}2)}{2\pi^{\frac{d+1}2}}\sin^{d-1}(\theta_1)d\theta_1\ldots\sin(\theta_{d-1})d\theta_{d-1}d\phi.
\end{eqnarray}
A natural cylindric structure is obtained by choosing $K=\mathrm{SO}(2)$, the circle subgroup responsible for rotation in the longitudinal variable $\phi$. The subset $G/H^{(K)}$ here contains only the two poles - $\theta_j=0$, $j=1,\ldots,d-1$, and $\theta_j=\pi$, $j=1,\ldots,d-1$, respectively. Thus indeed $\left|G/H^{(K)}\right|_\mu=0$, and Proposition \ref{CylStrCompHomSpProp} applies. The section $Z\subset G/H$ appearing in the proof of Proposition \ref{CylStrCompHomSpProp} can be chosen to correspond to the meridian $\phi=0$ in $\mathbb{S^d}$, which is Borel isomorphic to $\mathrm{SO}(2)\backslash\mathrm{SO}(d+1)/\mathrm{SO}(d)$ through the quotient map $\operatorname{q}$. In this way we have the almost isomorphism
\begin{equation}
\varphi:\mathrm{SO}(2)\times\mathrm{SO}(2)\backslash\mathrm{SO}(d+1)/\mathrm{SO}(d)\to \mathbb{S}^d
\end{equation}
given by $\varphi(\phi,(\theta_1,\ldots,\theta_{d-1}))=(\theta_1,\ldots,\theta_{d-1},\phi)$, i.e., simply by separation of the variable $\phi$. As a final step we parameterize $\mathrm{SO}(2)=\mathbb{S}^1$ by $t=\phi/2\pi$ to obtain an almost isomorphism
\begin{equation}
[0,1]\times\mathrm{SO}(2)\backslash\mathrm{SO}(d+1)/\mathrm{SO}(d)\to \mathbb{S}^d.
\end{equation}
This is the cylindric structure used implicitly in \cite{Gri90}.

\section*{Construction of $E$ and $g$}

The proof of the main theorem above is constructive, although the construction of the set $E$ and of the approximating function $g$ may be hard to follow due to the complexity of the proof. In this last section we will very briefly sketch that construction step by step.

\begin{itemize}

\item Choose an arbitrary ordering $\{R_k\}_{k=1}^\infty$ of all Fourier polynomials with rational coefficients.

\item For every $k\in\mathbb{N}$, choose a partition $\{\Delta_l(k)\}_{l=1}^{\nu_0(k)}$ of the cylindric measure space $\mathcal{M}=[0,1]\times\mathcal{N}$ of the form $\Delta_l(k)=[a_l(k),b_l(k)]\times\tilde\Delta_l(k)$ such that the measures $|\Delta_l(k)|$ are small enough, as well as a subordinate real step function $\Lambda(k)=\sum_{l=1}^{\nu_0(k)}\gamma_l(k)\chi_{\Delta_l(k)}$ ($\chi_X$ is the indicator function of the subset $X$), such that $\|\Lambda(k)-R_k\|_1$ is sufficiently small.

\item For every $k\in\mathbb{N}$, choose a number $\delta_*(k)\in(0,\frac12)$ so that $\{\delta_*(k)\}_{k=1}^\infty$ decays sufficiently rapidly. Define the periodic step function
    $$
    I(t)=1-\frac1{\delta_*(k)}\chi_{[0,\delta_*(k))}(t\mod 1),
    $$
    and the measurable function $\hat g_l^k\in L^\infty(\mathcal{M})$ by
    $$
    \hat g_l^k(x)=\gamma_l(k)I(s_0(k)t)\chi_{\Delta_l(k)}(x),\quad\forall x\in\mathcal{M},
    $$
    where the positive number $s_0(k)$ is sufficiently large. Define the measurable subsets $\hat E_l(k)\subset\Delta_l(k)$ by $$
    \hat E_l(k)=\left\{x\in\Delta_l(k)\,\vline\quad\hat g_l^k(x)=\gamma_l(k)\right\}.
    $$
    Define inductively the natural numbers $\hat N_l(k)$, $l=0,\ldots,\nu_0(k)$ and Fourier polynomials $\hat Q_l^k$, $l=1,\ldots,\nu_0(k)$, by setting $\hat N_0(1)=1$, $\hat N_0(k)=\hat N_{\nu_0(k)}(k-1)$ for $k>1$, and $\hat Q_l^k=\sum_{n=\hat N_{l-1}(k)}^{\hat N_l(k)-1}Y_n(\hat g_l^k)$ ( here $\{Y_n(g)\}_{n=1}^\infty$ is the Fourier series of the function $g\in L^2(\mathcal{M})$), so that the quantities
    $$
    \left\|\sum_{n=1}^{\hat N_l(k)-1}Y_n(\hat g_l^k)-\hat g_l^k\right\|_2
    $$
    are sufficiently small.

\item For every $k\in\mathbb{N}$, define the natural numbers $N_k=\hat N_{\nu_0(k)}-1$, measurable subsets $\tilde E_k=\bigcup_{l=1}^{\nu_0(k)}\hat E_l(k)$, and measurable functions $\tilde g_k\in L^\infty(\mathcal{M})$ by
    $$
    \tilde g_k=R_k-\Lambda(k)+\sum_{l=1}^{\nu_0(k)}\hat g_l^k,
    $$
    as well as Fourier polynomials
    $$
    \tilde Q_k=\sum_{l=1}^{\nu_0(k)}\hat Q_l^k=\sum_{n=N_{k-1}}^{N_k-1}\tilde Y_n.
    $$

\item Set
$$
E=\bigcap_{k=1}^\infty\tilde E_k.
$$

\item Choose by Lemma \ref{PolynomSeriesLemma} a subsequence $\{R_{k_s}\}_{s=0}^\infty$ such that $\|R_{k_s}\|_1$ decay sufficiently rapidly, and $\sum_{s=0}^\infty R_{k_s}=f$ in $L^1(\mathcal{M})$.

\item Define inductively the sequence of natural numbers $\{\nu_s\}_{s=1}^\infty$, $\nu_s>\nu_{s-1}$ for $s>1$, and measurable functions $g_s\in L^\infty(\mathcal{M})$ by choosing $\nu_1$ so that $N_{\nu_1-1}>\max\sigma(R_{k_0})$ and
    $$
    \left\|R_{\nu_s}-R_{k_s}+\sum_{j=1}^{s-1}\left[\tilde Q_{\nu_j}-g_j\right]\right\|_1
    $$
    is sufficiently small, and setting $g_s=R_{ks}+\tilde g_{\nu_s}-R_{\nu_s}$.

\item Finally, set
$$
g=R_{k_0}+\sum_{s=1}^\infty g_s.
$$

\end{itemize}

\section*{Acknowledgements}

The first named author thanks Prof. G. Folland for enlightening comments on topological double coset spaces.

This work, for the second named author, was supported by the RA MES
State Committee of Science, in the frames of the research project \textnumero\, 18T-1A148.
The third author was supported in parts by the FWO Odysseus Project, EPSRC grant EP/R003025/1 and by the Leverhulme
Grant RPG-2017-151.


\begin{thebibliography}{9}

\bibitem{Bar64}
N. Bary.
\textit{A Treatise on Trigonometric Series}.
Pergamon Press, 1964.

\bibitem{BoCl73}
A. Bonami, J. Clarc.
\textit{Sommes de C\'esar et multipl. des d\'eveloppements et harmoniques}.
Trans. Amer. Math. Soc. 183, 1973, Pages 223-263.


\bibitem{Fol15}
G. Folland.
\textit{ A Course in Abstract Harmonic Analysis}.
Taylor \& Francis Inc, 2015.

\bibitem{Gri90}
M. Grigorian.
\textit{On the convergence of Fourier-Laplace series} (in Russian).
Doklady AS USSR, 1990, Volume 315, Number 2, Pages 286-289.

\bibitem{Gri90a}
M. Grigorian.
\textit{On the convergence in $L^1$ metric and almost everywhere of Fourier series of complete orthonormal systems}. (in Russian).
Mat. Sbornik, 1990, Volume 181, Issue 8, Pages 1011-1030.

\bibitem{Gri91}
M. Grigorian.
\textit{On the convergence of Fourier series in the metric of $L^1$}.
Analysis Math., 1991, Volume 17, Pages 211-237.

\bibitem{GrKr13}
M. Grigoryan, V. Krotov.
\textit{Luzin's Correction Theorem and the Coefficients of Fourier Expansions in the Faber-Schauder System}.
Mat. Zametki 93(2), 2013, Pages 172-178.

\bibitem{GGK15}
M. Grigorian, L. Galoyan, and A. Kobelyan.
\textit{Convergence of Fourier series in classical systems}.
Sb. Math., Volume 206(7), 2015, Pages 941–979.

\bibitem{GrSa16}
M. Grigoryan, A. Sargsyan.
\textit{On the  universal  functions for the class  $L^p[0,1]$}.
J. Func. Anal., Volume 270, 2016, Pages 3111-3133.

\bibitem{GrGa16}
M. Grigoryan, L. Galoyan.
\textit{On the uniform convergence of negative order Cesaro means of Fourier series}.
J. Math. Anal. Appl., Volume 434(1), 2016, Pages 554-567.

\bibitem{Hel49}
S. Heladze.
\textit{Convergence of Fourier series almost everywhere and in the L-metric}.
Mat. Sb. 107(149):2(10), 1978, Pages 245–258; English transl.: Math. USSR-Sb. (35:4), 1979, Pages 527–539.

\bibitem{HeRo70}
E. Hewitt, K. Ross.
\textit{Abstract Harmonic Analysis. Vol II}.
Springer, 1970.

\bibitem{KaKo88}
B. Kashin, G. Kosheleva.
\textit{An approach to ‘correction’ theorems}.
Vestn. Moskov. Univ. Ser. 1 Mat. Mekh. (4), 1988, Pages 6–9; English transl.: Moscow Univ. Math. Bull. (43:4), 1988, Pages 1–5.

\bibitem{Kis79}
S. Kislyakov.
\textit{Quantitative aspect of correction theorems}.
Zap. Nauchn. Sem. LOMI, 1979, Volume 92, Pages 182–191.

\bibitem{Luz12}
N. Luzin. \textit{On the fundamental theorem of the integral calculus}.
Sbornik: Math. 28, 1912, Pages 266-294 (in Russian).

\bibitem{Men41}
D. Menchoff.
\textit{Sur la repr\'esentation des fonctions mesurables par des s\'eries trigonom\'etriques}.
Rec. Math. [Mat. Sbornik] N.S., 1941, Volume 9(51),	Number 3,	Pages 667–692.

\bibitem{Men51}
D. Menshov.
\textit{On Fourier series of continuous functions}.
Matematika, vol. IV, Uch. Zapiski Mosk. Gos. Univ., vol. 148, Moscow Univ. Press, 1951, Pages 108–132 (in Russian).

\bibitem{Osk76}
K. Oskolkov.
\textit{Uniform modulus of continuity of summable functions on sets of positive measure}.
Dokl. Akad. Nauk SSSR 229:2, 1976, Pages 304–306 (in Russian); English transl.: Soviet Math. Dokl. 17, 1976, Pages 1028–1030.

\bibitem{Pri69}
J. Price.
\textit{Walsh series and adjustment of functions on small sets}.
Illinois J. Math., Volume 13, Issue 1, 1969, Pages 131-136.

\bibitem{RuTu09}
M. Ruzhansky, V. Turunen.
\textit{Pseudo-Differential Operators and Symmetries}.
Birkh\"auser, 2010.

\bibitem{VK93}
N. Ja. Vilenkin, A. U. Klimyk.
\textit{Representation of Lie groups and special functions. Vol. 2. Class I representations, special functions, and integral transforms.}   Mathematics and its Applications (Soviet Series), 74. Kluwer Academic Publishers Group, Dordrecht, 1993.




\end{thebibliography}
\end{document}